\documentclass[letter,10pt,leqno, english]{amsart}
       
\usepackage{pdfsync}
\usepackage{cite}
\usepackage{calc}
\usepackage{enumerate,amssymb}
\usepackage{color}
\usepackage[arrow,curve,matrix,tips,2cell]{xy}
  \SelectTips{eu}{10} \UseTips
  \UseAllTwocells
\usepackage{tikz}
\usepackage{pdfcolmk}

\DeclareMathAlphabet{\matheurm}{U}{eur}{m}{n}

  \newcommand{\IR}{\mathbb{R}}

 \newcommand{\gr}{\mathfrak}
 \newcommand{\g}{{\mathfrak g}}

\newcounter{commentcounter}

\theoremstyle{plain}
\newtheorem{theorem}{Theorem}

\newtheorem{corollary}{Corollary}

\newtheorem*{theorem*}{Theorem}
\newtheorem*{mtheorem*}{Main Theorem}

\theoremstyle{definition}

\newtheorem{remark}{Remark}

\theoremstyle{remark}
\newtheorem*{summary*}{Summary}

\makeatletter\let\c@equation=\c@theorem\makeatother

\newcommand{\version}[1] 
{\begin{center} Last edited on #1\\
    Last compiled on \today\\
file name: \jobname
  \end{center}
}

 \title{On the equivarant de Rham cohomology for non-compact Lie groups}
        \author{Camilo Arias Abad}
        \address{Escuela de Matem\'aticas, Universidad Nacional de Colombia, Calle 59A No 63 - 20, Medell\'in, Colombia }   
      \email{camiloariasabad@gmail.com}
      \author{Bernardo Uribe}
      \address{Departamento de Matem\'aticas y Estad\'istica, Universidad del Norte, Km. 5 v\'ia Puerto Colombia, Barranquilla, Colombia}
 \email{buribe@gmail.com}
      \urladdr{https://sites.google.com/site/bernardouribejongbloed/}
       \subjclass[2010]{}
\thanks{Both authors acknowledge the financial support of the Alexander Von Humboldt Foundation and the Max Planck Institut for Mathematics.
The first author acknowledges the support of the SNSF and the hospitality of the University of Toronto. The second author acknowledges the financial support of COLCIENCIAS through grant number 121565840569 of the Fondo Nacional de Financiamiento para la Ciencia, la Tecnolog\'ia y la Innovaci\'on, Fondo Francisco Jos\'e de Caldas.}

\begin{document}

\begin{abstract}
Let $G$ be a connected and non-necessarily compact Lie group acting on a connected manifold $M$. In this short note we announce the following result: for a $G$-invariant closed differential form on $M$, the existence of a closed equivariant
extension in the Cartan model for equivariant cohomology is equivalent to the existence of an extension in the homotopy quotient.
  \end{abstract}

\maketitle

\newlength{\origlabelwidth} \setlength\origlabelwidth\labelwidth

\section{Introduction}

The Cartan model for the equivariant cohomology of the manifold $M$
$$\Omega_G^*M:= (S(\g^*) \otimes \Omega^*M)^G, \ \ d_G= d + \Omega^a\iota_{X_a}$$ can be seen as the 
de Rham version for the equivariant cohomology. Whenever the Lie group $G$ is compact, Cartan proved an equivariant
version of the De-Rahm Theorem, stating that the cohomology of the Cartan complex is canonically isomorphic
to the cohomology with real coefficients of the homotopy quotient  $H*(\Omega_G^*M) \cong H^*(M \times_G EG; \IR)$ \cite{Cartan} cf. \cite[Thm. 2.5.1]{Guillemin-Sternberg}. When the Lie group $G$ is not compact, the cohomology of the complex $\Omega_G^*M$ (which we also call the Cartan complex) fails in many situations to be isomorphic
to the cohomology of the homotopy quotient, and the explicit relation between the two has been very scarcely addressed.

Nevertheless, the Cartan complex is very well suited for studying equivariant conditions at the infinitesimal level.
Of particular interest is the study of the 
conditions under which there is absence of anomalies in gauged WZW actions on Lie groups. In \cite{Witten} Witten showed that the absence of anomalies in gauged WZW actions on compact
Lie groups was equivalent to the existence of closed equivariant extension of the WZW term on the Cartan complex, further showing that the existence or absence of anomalies is
purely topological. The arguments of Witten could be extended
without trouble to the non-compact case (see \cite[Chapter 4]{Uribe}), and together with the main result of this paper, we conclude that the absence or existence of anomalies is purely topological fact, independent of the compacity of the Lie group.

In this short note we investigate the relation between the cohomology
of the $G$-equivariant Cartan complex of $M$ and the cohomology of the homotopy quotient $M \times_G EG$, and we show
that indeed there is a surjective map from the former to the latter. In particular this result implies that 
 for a $G$-invariant closed differential form on $M$, the existence of a closed equivariant
extension in the Cartan model for equivariant cohomology is equivalent to the existence of an extension in the homotopy quotient.

 \section{Equivariant Cartan complex for connected Lie groups}
 
 Let $G$ be a connected Lie group with lie algebra $\g$. Let $K \subset G$ be a maximal compact subgroup of $G$ and denote
  by $\gr k$  its Lie algebra. The inclusion
 of Lie algebras ${\gr k} \hookrightarrow \g$ induces a dual map $\g^* \to {\gr k}^*$ which is ${\gr k}$-equivariant. Therefore
 we have the $K$-equivariant map
 $$S(\g^*) \to S(\gr k^*)$$
 from the symmetric algebra on $\g^*$ to the symmetric algebra on $\gr k^*$.
 
 Consider a manifold $M$ endowed with an action of  $G$. The Cartan complex associated to the $G$-manifold $M$ is
 $$\Omega_G^*M:= (S(\g^*) \otimes \Omega^*M)^G, \ \ d_G= d + \Omega^a\iota_{X_a}$$
 where $a$ runs over a base of $\g$, $\Omega^a$ denotes the element in $\g^*$ dual to $a$ and $X_a$ is the vector field 
 on $M$ that defines the element $a \in \g$.
\begin{remark}
In the literature, whenever the Cartan complex is used, it is assumed that the Lie group is compact. In this note 
we extend the notation of Cartan to the non-compact case.
\end{remark}
The composition of the natural maps
$$ (S(\g^*) \otimes \Omega^*M)^G \hookrightarrow (S(\g^*) \otimes \Omega^*M)^K \to (S(\gr k^*) \otimes \Omega^*M)^K$$
induces a homomorphism of Cartan complexes
$$\Omega_G^* M \to \Omega_K^*M.$$

 \begin{theorem}
 Let $G$ be a connected Lie group with Lie algebra $\g$, let $\gr k$ be the Lie algebra
 of the maximal compact subgroup $K$ of $G$ and consider a $G$-manifold $M$. Then the map
  $$\Omega_G^* M \to \Omega_K^*M$$
 induces a surjective map in cohomology
 $$H^*(\Omega_G^* M, d_G) \twoheadrightarrow H^*(\Omega_K^* M, d_K).$$
Since there are canonical isomorphisms $ H^*(\Omega_K^* M, d_K)\cong H^*(M \times_K EK,\IR) \cong H^*(M \times_G EG,\IR)$, we conclude that the canonical map
 $$H^*(\Omega_G^* M, d_G) \twoheadrightarrow H^*(M \times_G EG,\IR)$$
 is surjective.
\end{theorem}
 
 \begin{proof}

Consider the complex $C^k(G,  S({\gr g}^*) \otimes \Omega^\bullet M
)$ defined in \cite[Section 2.1]{Getzler} whose elements are smooth maps
$$f(g_1, \dots , g_k | X) : G^k \times \gr{g} \to \Omega^\bullet M, $$
which vanish if any of the arguments $g_i$ equals the identity of
$G$. The differentials $d$ and $\iota$ are defined by the formulas
\begin{eqnarray*}
 (df)(g_1, \dots , g_k | X) &=& (-1)^k df(g_1, \dots , g_k | X) \ \ \ \ \ \ {\rm{and}}\\
(\iota f) (g_1, \dots , g_k | X) &=& (-1)^k \iota(X) f(g_1, \dots
, g_k | X),
\end{eqnarray*} as in the case of the differentials in Cartan's
model for equivariant cohomology \cite{Cartan, Guillemin-Sternberg}. 

The differential $\bar{d}: C^k \to C^{k+1}$ is defined by the formula
\begin{eqnarray*}
(\bar{d}f)(g_0, \dots , g_k|X) & = & f( g_1, \dots , g_k | X ) +
 \sum_{i=1}^k (-1)^i f(g_0, \dots, g_{i-1}g_i, \dots  , g_k | X)\\
 & & +(-1)^{k+1} g_k f(g_0, \dots , g_{k-1} | {\rm{Ad}}(g_k^{-1})X),
\end{eqnarray*}
and the fourth differential $\bar{\iota} : C^k \to C^{k-1}$ is defined by
the formula
\begin{eqnarray*}
(\bar{\iota}f)(g_1, \dots , g_{k-1}|X) & = & \sum_{i=0}^{k-1} (-1)^i
\frac{\partial}{\partial t}  f(g_1, \dots, g_i, e^{tX_i}, g_{i+1}
\dots  , g_{k-1} | X),
\end{eqnarray*}
where $X_i= {\rm{Ad}}(g_{i+1} \dots g_{k-1})X$.

If the map
$$ f: G^k \to  S( \gr g^*) \otimes \Omega^\bullet M $$
has for image a homogeneous polynomial of degree $l$, then the total degree
of the map $f$ is $deg(f)=k+l$. The structural
maps $d, \iota, \bar{d}$ and $\bar{\iota}$ are all of degree 1, and
the operator $$d_G = d + \iota +\bar{d} + \bar{\iota}$$ becomes a
degree 1 map that squares to zero.

The cohomology of the complex
$$\left( C^*(G,
S({\gr g}^*) \otimes \Omega^\bullet M ) , d_G \right)$$ will be denoted by
 $$H^*(G, S({\gr g}^*) \otimes \Omega^\bullet M )$$
and in \cite[Thm. 2.2.3]{Getzler} it was shown that there is a canonical isomorphism of rings
$$H^*(G, S({\gr g}^*) \otimes \Omega^\bullet M ) \cong H^*(M \times_G
EG ; \IR)$$

Note that there are natural maps of complexes
 $$C^*(G, S(\g^*) \otimes \Omega^*M)  \to C^*(K,S(\gr k^*)\otimes \Omega^*M)$$
 inducing an isomorphism on cohomology groups
 $$H^*(G, S(\g^*) \otimes \Omega^*M)  \stackrel{\cong}{\to} H^*(K,S(\gr k^*) \otimes \Omega^*M).$$
 This isomorphism follows from the fact that the inclusion $K  \subset G$ is a homotopy equivalence inducing
 a homotopy equivalence $$M\times_KEK \simeq  M \times_G EG$$ and the fact that $$H^*(M\times_G EG,\IR) \cong H^*(G, S(\g^*)\otimes \Omega^*M)$$ for any
 connected Lie group $G$.
 
 Filtering the double complex $C^*(G, S(\g^*)\otimes \Omega^*M)$ by the degree of the elements in $S(\g^*)\otimes \Omega^*M$ we obtain a spectral
 sequence whose first page is 
 $$E_1=H^*_d(G,S(\g^*)\otimes \Omega^*M),$$ the differentiable cohomology of $G$ with values in the graded
 representation $S(\g^*)\otimes \Omega^*M$. Note that in the 0-th row we obtain
 $$E_1^{*,0}= (S(\g^*)\otimes \Omega^*M)^G=\Omega_G^*M.$$
  
 The same degree filtration applied to the complex $C^*(K,S(\gr k^*)\otimes \Omega^*M)$ produces a spectral sequence
 which at the first page is $\overline{E}_1=H^*_d(K,S(\gr k^*)\otimes \Omega^*M)$, and since $K$ is compact this simply becomes
 $$\overline{E}_1^{*,0}=(S(\gr k^*)\otimes \Omega^*M)^K=\Omega_K^*M$$
  with $\overline{E}_1^{p,q}=0$ for $q\neq 0$. 
 
 The first differential of the spectral sequence once restricted to the 0-th row $E_1^{*,0}=\Omega_G^*M$ is precisely
 the differential of the Cartan complex; therefore we obtain
 $$E_2^{*,0}= H^*(\Omega_G^*M).$$
 Equivalently we obtain
 $$\overline{E}_2^{*,0}= H^*(\Omega_K^*M)\cong H^*(M\times_K EK, \IR),$$
 but in this case the spectral sequence collapses at the second page and the only non zero elements
 in $\overline{E}_\infty$ appear on the 0-th row $\overline{E}_\infty^{*,0}\cong H^*(M\times_K EK, \IR)$.

 The canonical map between the complexes
$$C^*(G, S(\g^*) \otimes \Omega^*M)  \to C^*(K,S(\gr k^*)\otimes \Omega^*M)$$ induces a map of spectral sequences $E_\bullet \to \overline{E}_\bullet$, and we know
 that at the pages at infinity it should induce an isomorphism $E_\infty^{*,\star} \stackrel{\cong}{\to} \overline{E}_\infty^{*,\star} $. Therefore
 the map
 $$E_2^{*,0} \to \overline{E}_2^{*,0}$$
 must be a surjective map, and hence we have the canonical map
 $$\Omega_G^*M = E_1^{*,0} \to  \overline{E}_1^{*,0}= \Omega_K^*M$$
 inducing the desired
 surjective map in cohomology
 $$H^*(\Omega_G^* M, d_G) \twoheadrightarrow H^*(\Omega_K^* M, d_K).$$
 \end{proof}

 Finally, from the previous theorem  we may conclude:
 
 \begin{corollary}
 For a $G$-invariant closed differential form on $M$, the existence of a closed equivariant
extension in the Cartan model for equivariant cohomology is equivalent to the existence of an extension in the homotopy quotient.
 \end{corollary}
 \begin{proof}
 A $G$-invariant closed differential form on $M$ may be extended to a closed form in the Cartan complex 
 if and only if its cohomology class lies in the image of the projection map
 $$H^*(\Omega_G^* M, d_G) \to H^*(M).$$
 
 This projection map can be seen as the composition of the maps
  $$H^*(\Omega_G^* M, d_G) \twoheadrightarrow H^*(\Omega_K^* M, d_K) \to H^*(M).$$
 Since the left hand side map is surjective, a  $G$-invariant closed differential form on $M$ 
  may be extended to a closed form in the Cartan complex
 if and only if its cohomology class lies in the image of the right hand side map.
The canonical isomorphisms 
$$H^*(M\times_GEG; \IR) \cong H^*(M\times_KEK; \IR) \cong H^*(\Omega_K^* M, d_K)$$
imply the result.
 \end{proof}


\bibliography{Cartan-non-compact-Lie}

\begin{thebibliography}{1}

\bibitem{Cartan}
Henri Cartan.
\newblock La transgression dans un groupe de {L}ie et dans un espace fibr\'e
  principal.
\newblock In {\em Colloque de topologie (espaces fibr\'es), {B}ruxelles, 1950},
  pages 57--71. Georges Thone, Li\`ege; Masson et Cie., Paris, 1951.

\bibitem{Uribe}
Hugo Garc\'ia-Compe\'an, Pablo Paniagua, and Bernardo Uribe.
\newblock Equivariant extensions of differential forms for non-compact lie
  groups.
\newblock In {\em The influence of {S}olomon {L}efschetz in Geomerty and
  Topology- 50 years of {CINVESTAV}}, volume 621 of {\em Contemp. Math.}, pages
  19--33. Amer. Math. Soc., Providence, RI, 2014.

\bibitem{Getzler}
Ezra Getzler.
\newblock The equivariant {C}hern character for non-compact {L}ie groups.
\newblock {\em Adv. Math.}, 109(1):88--107, 1994.

\bibitem{Guillemin-Sternberg}
Victor~W. Guillemin and Shlomo Sternberg.
\newblock {\em Supersymmetry and equivariant de {R}ham theory}.
\newblock Mathematics Past and Present. Springer-Verlag, Berlin, 1999.
\newblock With an appendix containing two reprints by Henri Cartan [ MR0042426
  (13,107e); MR0042427 (13,107f)].

\bibitem{Witten}
Edward Witten.
\newblock On holomorphic factorization of {WZW} and coset models.
\newblock {\em Comm. Math. Phys.}, 144(1):189--212, 1992.

\end{thebibliography}
\bibliographystyle{plain}
\end{document}